\documentclass[11pt,a4paper]{article}

\usepackage{amsmath,amsfonts,amssymb}
\usepackage{graphicx}
\usepackage{url}

\usepackage[text={14cm,22cm},centering]{geometry}
\renewcommand{\baselinestretch}{1.3} % note: also appendix

\usepackage{amsthm}

\newtheorem{theorem}{Theorem}
\newtheorem{lemma}[theorem]{Lemma}

\newcommand{\N}{\mathbb{N}}
\newcommand{\Prob}{{\mathbb{P}}}

\title{Approximate results for a generalized secretary problem}
\author{%
Chris Dietz, Dinard van der Laan, Ad Ridder\thanks{
Corresponding author's address:
Ad Ridder, Department of Econometrics and Operations Research,
Vrije University Amsterdam, The Netherlands;
email: {\tt aridder@feweb.vu.nl}}\\[0.2cm]
{\it Vrije University, Amsterdam, Netherlands}\\[0.2cm]
email \url{{cdietz,dalaan,aridder}@feweb.vu.nl}
}
\begin{document}
\maketitle

\begin{abstract}

\noindent A version of the classical secretary problem is studied, in which one is interested in selecting one of the
$b$ best out of a group of $n$ differently ranked persons who are presented one by one in a random order. It is assumed
that $b\geq 1$ is a preassigned number. It is known, already for a long time, that for the optimal policy one needs to
compute $b$ position thresholds, for instance via backwards induction. In this paper we study approximate policies,
that use just a single or a double position threshold, albeit in conjunction with a level rank. We give exact and
asymptotic (as $n\to\infty$) results, which show that the double-level policy is an extremely accurate approximation.
\end{abstract}


%%%%%%%%%%%%%%%%%% Document start %%%%%%%%%%%%%%%%%%%%%

\section{Introduction}\label{s:intro}
The classical secretary problem is a well known optimal stopping problem
from probability theory. It is usually described by different
real life examples, notably the process of hiring a secretary.
Imagine a company manager in need of a secretary.
Our manager wants to hire only the best secretary from
a given set of $n$ candidates, where $n$ is known. No candidate
is equally as qualified as another. The manager decides to
interview the candidates one by one in a random fashion.
Every time he has interviewed a candidate he has to decide
immediately whether to hire her or to reject her
and interview the next one.
During the interview process he can only judge
the qualities of those candidates he has already interviewed.
This means that for every candidate he has observed, there might
be an even better qualified one within the set of candidates yet
to be observed. Of course the idea is that by the time only a small
number of candidates remain unobserved, a recently interviewed
candidate that is relatively best will probably also be the overall
best candidate.
\par\bigskip\noindent
There is abundant research literature on this classical secretary problem,
for which we refer to Ferguson \cite{ferguson} for an historical note and an extensive bibliography.
The exact optimal policy is known, and may be derived by various methods, see for instance
Dynkin and Yushkevich \cite{dynkinyushkevich}, and Gilbert and Mosteller \cite{gilbertmosteller}.
Also, many variations and generalizations of the original problem have
been introduced and analysed. One of these generalizations is the
focus of our paper, namely the problem to select one of the $b$ best,
where $1\leq b\leq n$ is some preassigned number
(notice that $b=1$ is the classical secretary problem). Originally, this
problem was introduced by Gusein-Zade \cite{guseinzade}, who derived the structure
of the optimal policy: there is a sequence $0\leq s_1<s_2<\cdots<s_b\leq s_{b+1}=n-1$ of
position thresholds such that when candidate $i$ is presented, and
judged to have relative rank $k$ among the first $i$
candidates\footnote{It is most convenient to rank the candidates $1,2,\ldots,n$, with
rank 1 being the best, rank 2 being second best, etc.},
then the optimal decision says
\begin{align*}
& i\leq s_1: \; \mbox{continue whatever $k$ is};\\
& s_j+1 \leq i \leq s_{j+1}\; (\mbox{where $j=1,\ldots,b$}): \;
\begin{cases}
\mbox{stop if $k\leq j$}\\
\mbox{continue if $k>j$};
\end{cases}\\
& i=n: \; \mbox{stop whatever $k$ is}.
\end{align*}
Furthermore, \cite{guseinzade} gave an  algorithm
to compute these thresholds, and derived
asymptotic expressions (as $n\to\infty$) for the $b=2$ case.
Also Frank and Samuels \cite{franksamuels} proposed an algorithm,
and gave the limiting (as $n\to\infty$) probabilities and
limiting proportional thresholds $s_j/n$.
\par\bigskip\noindent
The algorithms of \cite{franksamuels,guseinzade} are based on dynamic programming, which
means that the optimal thresholds $s_j$, and the optimal winning probability
are determined numerically. The next interest was to find analytic expressions.
To our best knowledge, this has been resolved only for $b=2$ by
Gilbert and Mosteller \cite{gilbertmosteller},
and for $b=3$ by Quine and Law \cite{quinelaw}. Although the latter claim that their approach is
applicable to produce exact results for any $b$, it is clear that the expressions
become rather untractable for larger $b$. This has inspired us to
develop approximate results for larger $b$.
\par\bigskip\noindent
We consider two approximate policies for the general $b$ case: single-level policies,
and double-level policies. A single-level policy is given by a single position
threshold $s$ in conjuction with a rank level $r$, such that
when candidate $i$ is presented, and
judged to have relative rank $k$ among the first $i$
candidates, then the policy says
\begin{align*}
& i\leq s: \; \mbox{continue whatever $k$ is};\\
& s+1 \leq i \leq n-1\; : \;
\begin{cases}
\mbox{stop if $k\leq r$}\\
\mbox{continue if $k>r$};
\end{cases}\\
& i=n: \; \mbox{stop whatever $k$ is}.
\end{align*}
A double-level policy is given by two position
thresholds $s_1<s_2$ in conjuction with two rank levels $r_1<r_2$, such that
when candidate $i$ is presented, and
judged to have relative rank $k$ among the first $i$
candidates, then the policy says
\begin{align*}
& i\leq s_1: \; \mbox{continue whatever $k$ is};\\
& s_1+1 \leq i \leq s_2\; : \;
\begin{cases}
\mbox{stop if $k\leq r_1$}\\
\mbox{continue if $k>r_1$};
\end{cases}\\
& s_2+1 \leq i \leq n-1\; : \;
\begin{cases}
\mbox{stop if $k\leq r_2$}\\
\mbox{continue if $k>r_2$};
\end{cases}\\
& i=n: \; \mbox{stop whatever $k$ is}.
\end{align*}
We shall derive the exact winning probability for these two approximate policies,
when the threshold and level parameters are given. These expressions
can then used easily to compute the optimal single-level and the optimal
double-level policies, i.e., we optimize the winning probabilities
(under these level policies) with respect to their threshold and level parameters.
The most important result is that the winning probabilities of the optimal
double-level policies are extremely close to the winning probabilities of the optimal
policies (with the $b$ thresholds), specifically for larger $b$.
In other words, we have found explicit formulas that approximate closely
the winning probabilities for this generalized secretary problem.
As an example, we present in Table \ref{t:resdlp} the relative
errors in percentages for a few $n,b$ combinations (a more extended
table can be found in Section \ref{s:num}).
\begin{table}[htb]
\begin{center}
\caption{\textit{Relative errors (\%) of the optimal double-level policies.}}
\label{t:resdlp}
\medskip
\begin{tabular}{l | r r r }
& $n=100$ & $n=250$ & $n=1000$\\
\hline
$b=10$ & 1.702 & 1.841 & 1.911 \\
$b=25$ & 0.036 & 0.066 & 0.084
\end{tabular}
\end{center}
\end{table}
Our second contribution is the derivation of asymptotic results
for our level policies as $n\to\infty$. This means both for the (optimal) winning
probabilities, and the (optimal) fractional thresholds $s_j/n$.
\par\bigskip\noindent
The paper is organized as follows.
Section \ref{s:slp} contains the derivation
of the exact winning probability of the single-level policy,
and the associated asymptotic results;
Section \ref{s:dlp}  deals with the two-level policies. Finally,
in Section \ref{s:num}
we demonstrate how accurate the approximate policies are performing.

\section{Single-level policies}\label{s:slp}

Before we consider the single-level polcies we first introduce some notation we use throughout this paper. The absolute
rank of the $ i $-th object is denoted by $ X_i $, while the relative rank of the $ i $-th object is denoted by $ Y_i
$. Ranks run from $1$ to $n$, and we say that rank $i$ is higher than rank $j$ when $i<j$. Moreover for natural numbers
$ x $ and $ n $, the falling factorial $ x (x-1)\ldots (x-n+1) $ is denoted by $ (x)_n $. Note that $ (x)_n $ is the
number of $ n $-permutations of a set containing $x $ elements which is also the number of different injective
functions from $ \{1,2,\ldots,n\} $ to $ \{1,2,\ldots,x\}$. It is easily seen that $ (x)_n = n! \binom{x}{n} $ and thus
we have
\begin{equation}\label{idfalfac}
\frac{(x)_n}{(y)_n} = \frac{\binom{x}{n}}{\binom{y}{n}}.
\end{equation}

\subsection{Winning probability for single-level policies}\label{s:wslp}

In this variant of the well-known secretary problem the objective is to pick one of the $ b $ best objects from $ n $
objects consecutively arriving one by one in the usual random fashion known from the classical problem. In this
subsection we consider the performance of the class of so-called single-level policies which is determined by two
integer parameters $ s $ (called the position threshold) and $ r $ (called the rank level). Following such a
single-level policy objects are considered to be selected from position $s +1$ and then the first one encountered with
a relative rank higher or equal than $ r $ is picked. Moreover, we assume that if the first $ n-1 $ items are not
picked that then the last object is certainly picked independent of its relative rank $ Y_n $. Let $ \pi = \pi(s,r) $
be such a policy with $ 0 \leq s \leq n-1 $ and $ 1 \leq r \leq b $. To analyze the performance of this class of
policies an explicit expression for the probability $ P_{\rm SLP}(\pi) $ of success when applying the single-level
policy $ \pi = \pi(s,r) $ will be obtained. Thus $ P_{\rm SLP}(\pi) $ is the probability that an object is picked with
absolute rank higher than or equal to $ b $ if policy $ \pi $ is applied. Note: when we wish to express explicitly
parameters ($n,b,s,r$) we denote it, otherwise we omit it.
%This probability is easy obtained for the special case that $ s =n-1$.

\begin{theorem}\label{t:slpwinprob}
For $ r=1,2,\ldots,b $ we have that $P_{\rm SLP}(\pi(s,r)) = \frac{b}{n}$ if
$0\leq s \leq r-1$ or $s = n-1$, and otherwise
\begin{equation}
\label{e:slpwinprob}
%\begin{split}
P_{\rm SLP}(\pi(s,r)) = \sum_{i=s+1}^{n-1} \, \frac{(s)_r}{(i-1)_r}
\left( \frac{r}{n} + \frac1n\sum_{j=r+1}^b \sum_{k=1}^r
\frac{\binom{j-1}{k-1}\binom{n-j}{i-k}}{\binom{n-1}{i-1}}\right)
+ \frac{(s)_r}{(n-1)_r} \, \frac{b}{n}.
%\end{split}
\end{equation}
\end{theorem}
Before proving this expression we need two auxiliary results.
\begin{lemma}\label{l:minprob}
For $ s =0,1,\ldots,n-2 $ and $ i = s+2,s+3,\ldots,n $ we have that
\[
\Prob(\min\{Y_{s+1},Y_{s+2},\ldots,Y_{i-1}\} > r) =
\begin{cases}0 & \mbox{ if } s<r \\
\frac{(s)_r}{(i-1)_r} & \mbox{ if } s \geq r.
\end{cases}
\]
\end{lemma}
\begin{proof}
Let $ A $ be the event $ \min\{Y_{s+1},Y_{s+2},\ldots,Y_{i-1}\} > r $.
If $ s<r $ then $ Y_{s+1} \leq r $ and thus $\Prob(A) = 0 $. For $ s \geq r $, $ i=s+2,\ldots,n-1 $
we have that $ \Prob(A) $ is the probability that the rankings
$1,2,\ldots,r $ are contained in the first $ s $ positions of a random permutation of the
numbers $ 1,2,\ldots,i-1 $.
Thus $ \Prob(A) $ is the number of distinct injective functions from $ \{1,2,\ldots,r\} $
to $ \{1,2,\ldots,s\} $ divided
by the number of distinct injective functions from $ \{1,2,\ldots,r\} $
to $ \{1,2,\ldots,i-1\} $. Hence $ \Prob(A) =\frac{(s)_r}{(i-1)_r} $.
\end{proof}
\begin{lemma}\label{l:hypprob}
For $ i =s+1,s+2,\ldots,n-1 $ and $ r =1,2,\ldots,b $ we have that
\[ \Prob(Y_i \leq r |X_i = j) =
\begin{cases} 1 & \mbox{ for }  j=1,2,\ldots,r \\
\sum_{k=1}^r \frac{\binom{j-1}{k-1}\binom{n-j}{i-k}}{\binom{n-1}{i-1}}
 & \mbox{ for } j=r+1,r+2,\ldots,b.
\end{cases}
\]
\end{lemma}
\begin{proof}
We have that $ Y_i \leq X_i $ and thus $ \Prob(Y_i \leq r |X_i = j) = 1 $ if $ j \leq r $.
Suppose that $ j \in \{r+1,r+2,\ldots,b\} $. If $ X_i = j $ then $ Y_i = k $
for $ k \in \{1,2,\ldots,r\} $ if $ k - 1 $ objects from the $ j-1 $ objects with
absolute ranking smaller than $ j $ are among the first $ i-1 $ objects.
Thus $ \Prob(Y_i = k | X_i = j)= \Prob(H=k-1)$ where $ H $ is a hypergeometrically distributed
random variable with population size $ n-1 $, sample size
$ i-1 $ and $ j-1 $ successes in the population. Hence
\[
\Prob(Y_i \leq r |X_i = j) = \sum_{k=1}^r \Prob(H=k-1) = \sum_{k=1}^r
\frac{\binom{j-1}{k-1}\binom{n-j}{i-k}}{\binom{n-1}{i-1}}.
\]
\end{proof}

\begin{proof} (Of Theorem \ref{t:slpwinprob}.)
The case $s=n-1$ is trivial because then
$P_{\rm SLP}(\pi(n-1,r)) = \Prob(X_n \leq b) = \frac{b}{n}$.

The cases $s=0,\ldots,r-1$ are also trivial because then for sure $Y_{s+1}\leq r$,
and thus $P_{\rm SLP}(\pi(s,r)) = \Prob(X_{s+1} \leq b) = \frac{b}{n}$.

Now consider the `general' case.
For $ i = s+1,s+2,\ldots,n $ and $ j=1,2,\ldots,b $ let $ A^i_j $ be the event that $ X_i = j $
and policy $ \pi(s,r) $ picks the object at position $ i $:
\[ A^i_j = \{\min\{Y_{s+1},Y_{s+2},\ldots,Y_{i-1}\} > r, Y_i \leq r, X_i = j \}.
\]
Thus,
\begin{align*}
P_{\rm SLP}(\pi(s,r))&= \sum_{i=s+1}^n \sum_{j=1}^b \Prob(A^i_j)\\
&=\sum_{j=1}^b \Prob(A^{s+1}_j) + \sum_{i=s+2}^{n-1} \sum_{j=1}^b \Prob(A^i_j)
+\sum_{j=1}^b \Prob(A^n_j).
\end{align*}
Cases $i=s+1$ and $i=n$ are treated seperately.
Notice that for $ k < i $ the relative ranks $ Y_k $ are independent of both $ X_i $ and $ Y_i $,
thus (for $s+2\leq i\leq n-1$)
\begin{align*}
\Prob(A^i_j) &= \Prob(\min\{Y_{s+1},Y_{s+2},\ldots,Y_{i-1}\} > r , Y_i \leq r ,X_i = j)\\
& = \Prob(\min\{Y_{s+1},Y_{s+2},\ldots,Y_{i-1}\} > r )\, \Prob(Y_i \leq r ,X_i = j)\\
& = \Prob(\min\{Y_{s+1},Y_{s+2},\ldots,Y_{i-1}\} > r )\,\Prob(Y_i \leq r |X_i = j)\, \Prob(X_i=j),
\end{align*}
with $\Prob(X_i=j)=\frac1n$, and the other two factors were determined in Lemma \ref{l:minprob}
and  Lemma \ref{l:hypprob}.
For $i=s+1$:
\[
\Prob(A^{s+1}_j)= \Prob(X_{s+1} = j, Y_{s+1} \leq r)=
 \Prob(Y_{s+1} \leq r |X_{s+1} = j) \, \Prob(X_{s+1}=j),
 \]
and then apply Lemma \ref{l:hypprob} while noticing that $(s)_r/(i-1)_r=1$.
For $i=n$:
\begin{align*}
\Prob(A^n_j) & = \Prob(\min\{Y_{s+1},Y_{s+2},\ldots,Y_{n-1}\} > r, X_n = j)\\
&=
\Prob(\min\{Y_{s+1},Y_{s+2},\ldots,Y_{n-1}\} > r) \Prob(X_n = j),
\end{align*}
and apply Lemma \ref{l:minprob}.
\end{proof}
We defer the comparison of the performance of single-level policies with the optimal
policy to Section \ref{s:num}.
\subsection{Asymptotic results for single-level policies}\label{ss:slpass}
For any $ n \in \N $ the probability of successfully applying a single-level
policy $ \pi(s,r) $ is given by \eqref{e:slpwinprob}. Moreover, by enumeration,
marginal analysis, and/or
dynamic programming the values of $ s $ and $ r $ maximizing $ P_{\rm SLP}(\pi(s,r)) $
may be obtained (see Section \ref{s:num}), but in general the computation time increases
if $ n $ gets larger.
However, in the limit $ n \rightarrow \infty $ the expression given by \eqref{e:slpwinprob}
may be simplified.
In this section we will find asymptotic results on the performance of an important
family of single-level policies for $ n \rightarrow \infty $. To obtain these asymptotic
results we restrict to the family of single-level policies for
which there exists some $ n_0 \in \N $ , $ r_0 \in \{1,2,\ldots,b\} $ and $ 0<\alpha<1 $
such that $ r = r_0 $ for all $ n \geq n_0 $ and
$ \lim_{n \rightarrow \infty} \frac{s}{n} = \alpha $.
In other words we assume that the rank level $ r $ is fixed for large $ n $ while
the position threshold $ s \approx \alpha n $.
This is motivated by numerical evidence from dynamic programming that
optimal single-level policies have these asymptotical properties. Under this assumption
the optimal values of $ r_0 $ and $ \alpha $ should depend (only) on $ b $. The idea is
that by obtaining an asymptotic simplification of \eqref{e:slpwinprob} for all positive
integers $ b $ the corresponding asymptotical optimal values for $ r_0 $ and
$ \alpha $ could be obtained analytically.
\begin{theorem}\label{t:slpwinprobass}
Fix $r\in\{1,\ldots,b\}$, and let $s=\alpha n+o(n)$ (as $n\to\infty$), then
\begin{equation}\label{e:slpwinprobass}
\lim_{n\to\infty} P_{\rm SLP}^{(n)}(\pi(s,r)) =
\alpha^r\int_\alpha^1\frac{1}{x^r}\,\left(
r + \sum_{j=r+1}^b\sum_{k=1}^r \binom{j-1}{k-1}x^{k-1}(1-x)^{j-k}\right)\,dx.
\end{equation}
\end{theorem}
\begin{proof}
Consider the hypergeometric probability
$\binom{j-1}{k-1}\binom{n-j}{i-k}/\binom{n-1}{i-1}$ in the expression
of the winning probability \eqref{e:slpwinprob} (see also Lemma \ref{l:hypprob}).
Swapping the parameters of sample size and the number of successes, we know
that this probability is equal to
$\binom{i-1}{k-1}\binom{n-i}{j-k}/\binom{n-1}{j-1}$.
This is interpreted as the probability of finding $k-1$ successes
in a sample of size $j-1$ when this sample is drawn without replacement
from a population of size $n-1$ containing a total number of $i-1$ possible successes.
We notice in expression \eqref{e:slpwinprob} that both population size $n-1$
and the total number of $i-1$ successes tend to infinity, proportionally,
whereas the sample size $j-1$ remains fixed. Hence, we may approximate
the hypergeometric probability by a binomial probability:
\begin{equation}
\label{e:asshyper}
\begin{split}
& \frac{\binom{i-1}{k-1}\binom{n-i}{j-k}}{\binom{n-1}{j-1}} =
\binom{j-1}{k-1}\,\left(\frac{i-1}{n-1}\right)^{k-1}\,
\left(\frac{n-i}{n-1}\right)^{j-k} + o(1)\\
& \quad =\binom{j-1}{k-1}x^{k-1}\,(1-x)^{j-k} +o(1),
\end{split}
\end{equation}
where $i,n\to\infty$ such that $(i-1)/(n-1)\to x\in(0,1)$.
Notice that equivalently, $i/n\to x$.
\par\bigskip\noindent
Next, consider the ratio $(s)_r/(i-1)_r$:
\begin{equation}\label{e:assratio}
\begin{split}
& \frac{(s)_r}{(i-1)_r}
=\prod_{k=0}^{r-1}\frac{s-k}{i-1-k}
=\prod_{k=0}^{r-1}
\frac{\frac{s}{n}-\frac{k}{n}}{\frac{i}{n}-\frac{1+k}{n}}\\
&\quad
=\prod_{k=0}^{r-1}\frac{\alpha+o(1)}{\frac{i}{n}+o(1)}
= \left(\frac{\alpha}{x}\right)^r +o(1)\quad(n\to\infty).
\end{split}
\end{equation}
Combining the two asymptotics \eqref{e:asshyper} and \eqref{e:assratio},
it is easily seen that the first term
in the expression of the winning probability \eqref{e:slpwinprob}, i.e., the $\Sigma$ part,
can be considered as a Riemann sum
converging to the integral \eqref{e:slpwinprobass}.
Finally, the last term in the expression
of the winning probability \eqref{e:slpwinprob}
is clearly less than $b/n$, and thus converges to
zero as $n\to\infty$.
\end{proof}
Denote the integrand in expression \eqref{e:slpwinprobass}
of the asymptotic winning probability by $f(x)$. Thus,
the expression is
\[
P_{\rm SLP}^{(\infty)}(\pi(\alpha,r))= \alpha^r\int_\alpha^1f(x)\,dx.
\]
The function $\alpha\in(0,1)\mapsto P_{\rm SLP}^{(\infty)}(\pi(\alpha,r))$
is unimodal concave, thus making it easy to find numerically
the optimal $\alpha$ by solving the first-order condition
using the bisection procedure. Hence, we can compare empirically
the asymptotic optimal winning probability with
(finite) optimal winning probabilities. As an example, we set
$b=5, r=3$. Then we find $\alpha^*=0.5046$ and
$P_{\rm SLP}^{(\infty)}(\pi(\alpha^*,r))=0.765697$. Figure \ref{f:convb5r3}
shows the winning probabilities $P_{\rm SLP}^{(n)}(\pi(s^*,r))$,
where $s^*=s^*(r)$ denotes the optimal position threshold given rank level $r$ and
the number of candidates $n$, obtained by optimizing the
$P_{\rm SLP}^{(n)}(\pi(s,r))$ with respect to $s$ (see Section~\ref{s:num}).

\begin{figure}[htb]
\begin{center}
\includegraphics[width=8cm]{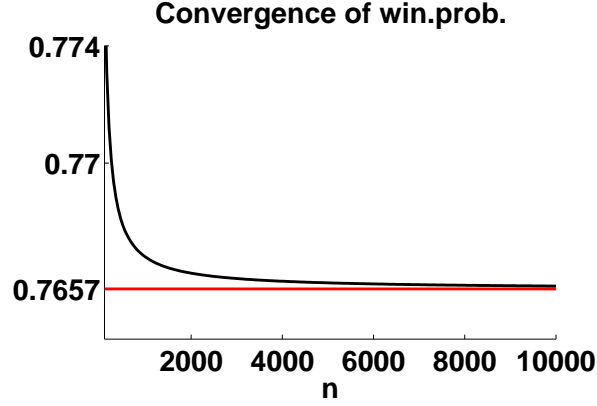}
\caption{\textit{Optimal finite and asymptotic winning probabilities for
single-level policies with $b=5, r=3$.}}
\label{f:convb5r3}
\end{center}
\end{figure}
\section{Double-level policies}\label{s:dlp}
A natural extension of the single-level policies is the class
of double-level policies for the secretary problem
where the objective is to pick one of the $ b $ best objects from $ n $ objects
consecutively arriving one by one in the usual random fashion.
Let be given two rank levels $1\leq r_1<r_2\leq b$,
and two position thresholds $r_1\leq s_1<s_2\leq n-1$
(we discard the trivial cases of $s_2=n$ which gives again a single-level
policy, and $s_1<r_1$ which leads to stopping at position $s_1+1$).
The double-level policy says to observe the first $s_1$ presented objects
without picking any; next, from objects at positions $s_1+1$ up to $s_2$
the first one encountered with a relative rank higher or equal
than $ r_1 $ is picked; if no such object appears,
the first object at positions $s_2+1$ up to $n-1$ is selected which
has a relative rank of at least $r_2$; finally, if all these $ n-1 $ items
are not picked, the last object is certainly picked independent of its
relative rank $ Y_n $. Slightly abusing, we denote again by $ \pi = \pi(s,r) $ such
a double-level policy and by $ P_{\rm DLP}(\pi) $ its winning probability.
\begin{theorem}\label{t:dlpwinprob}
The double-level policy given by rank levels $1\leq r_1<r_2\leq b$,
and position thresholds $r_1\leq s_1<s_2\leq n-1$
has winning probability
\begin{equation}\label{e:dlpwinprob}
\begin{split}
& P_{\rm DLP}(\pi(s,r)) = \sum_{i=s_1+1}^{s_2}\,\frac{(s_1)_{r_1}}{(i-1)_{r_1}}\,
\left(\frac{r_1}{n}\, + \, \frac{1}{n} \sum_{j=r_1+1}^b \sum_{k=1}^{r_1}
\frac{\binom{j-1}{k-1}\binom{n-j}{i-k}}{\binom{n-1}{i-1}}\right)\\
& + \sum_{i=s_2+1}^{n-1}\,\frac{(s_1)_{r_1}\,(s_2-r_1)_{r_2-r_1}}{(i-1)_{r_2}}\,
\left(\frac{r_2}{n}\, + \, \frac{1}{n} \sum_{j=r_2+1}^b \sum_{k=1}^{r_2}
\frac{\binom{j-1}{k-1}\binom{n-j}{i-k}}{\binom{n-1}{i-1}}\right)\\
& + \frac{(s_1)_{r_1}\,(s_2-r_1)_{r_2-r_1}}{(n-1)_{r_2}}\,\frac{b}{n}.
\end{split}
\end{equation}
\end{theorem}
\begin{proof}
Similar to the proof of Theorem \ref{t:slpwinprob}, once we have established that
\[
\Prob(\min\{Y_{s_1+1},\ldots,Y_{s_2}\} > r_1; \;
\min\{Y_{s_2+1},\ldots,Y_{i-1}\} > r_2) =
\frac{(s_1)_{r_1}\,(s_2-r_1)_{r_2-r_1}}{(i-1)_{r_2}}.
\]
This can be proved as follows. Let $A$ be the event of concern, then
$ \Prob(A) $ is the probability that the rankings
$1,2,\ldots,r_1 $ are contained in the first $ s_1 $ positions of a random permutation of the
numbers $ 1,2,\ldots,i-1 $, and the rankings
$r_1+1,r_1+2,\ldots,r_2 $ are contained in the first $ s_2 $ positions.
However, any permutation for which $1,2,\ldots,r_1 $ are contained in the first
$ s_1 $ positions, leaves $s_2-r_1$ positions for rankings $r_1+1,r_1+2,\ldots,r_2 $
in order to become a `feasible' permutation; the remaining $i-1-r_2$ rankings
can be positioned arbitrary. For $\Prob(A)$ we divide the number of
feasible permutations by the total number of permutations:
\[
\Prob(A) = \frac{(s_1)_{r_1}\,(s_2-r_1)_{r_2-r_1}\,(i-1-r_2)!}{(i-1)!}
=\frac{(s_1)_{r_1}\,(s_2-r_1)_{r_2-r_1}}{(i-1)_{r_2}}.
\]
\end{proof}
\begin{theorem}\label{t:dlpwinprobass}
Fix rank levels $1\leq r_1<r_2\leq b$, and let $s_1=\alpha_1 n+o(n)$
and $s_2=\alpha_2 n+o(n)$ (as $n\to\infty$), where $0<\alpha_1<\alpha_2<1$.
Then
\begin{equation}\label{e:dlpwinprobass}
\begin{split}
& \lim_{n\to\infty} P_{\rm DLP}^{(n)}(\pi(s,r)) =
\alpha_1^{r_1}\int_{\alpha_1}^{\alpha_2}\frac{1}{x^{r_1}}\,\left(
r_1 + \sum_{j=r_1+1}^b\sum_{k=1}^{r_1} \binom{j-1}{k-1}x^{k-1}(1-x)^{j-k}\right)\,dx\\
& \quad +
\alpha_1^{r_1}\alpha_2^{r_2-r_1}\int_{\alpha_2}^{1}\frac{1}{x^{r_2}}\,\left(
r_2 + \sum_{j=r_2+1}^b\sum_{k=1}^{r_2} \binom{j-1}{k-1}x^{k-1}(1-x)^{j-k}\right)\,dx.
\end{split}
\end{equation}
\end{theorem}
\begin{proof}
Similar to the proof of Theorem \ref{t:slpwinprobass}.
\end{proof}
Denote the two integrands in expression \eqref{e:dlpwinprobass}
of the asymptotic winning probability by $f_1(x)$ and $f_2(x)$, respectively. Thus,
the expression is
\[
P_{\rm DLP}^{(\infty)}(\pi(\alpha,r))= \alpha_1^{r_1}\int_{\alpha_1}^{\alpha_2}f_1(x)\,dx
+ \alpha_1^{r_1}\alpha_2^{r_2-r_1}\int_{\alpha_2}^{1}f_2(x)\,dx.
\]
The functions
$\alpha_1\in(0,\alpha_2)\mapsto P_{\rm DLP}^{(\infty)}(\pi(\alpha,r))$ (keeping $\alpha_2$ fixed),
and $\alpha_2\in(\alpha_1,1)\mapsto P_{\rm DLP}^{(\infty)}(\pi(\alpha,r))$ (keeping $\alpha_1$ fixed)
are unimodal concave, thus making it easy to find numerically
the optimal $\alpha=(\alpha_1,\alpha_2)$ by solving the first-order conditions
using the bisection procedure. Hence, we can compare empirically
the asymptotic optimal winning probability with
(finite) optimal winning probabilities. As an example, we set
$b=10, r_1=2, r_2=6$. Then we find $\alpha_1^*=0.3630$,
$\alpha_2^*=0.6446$ and $P_{\rm DLP}^{(\infty)}(\pi(\alpha^*,r))=0.957643$.
Figure \ref{f:convb10r12r26}
shows the winning probabilities $P_{\rm DLP}^{(n)}(\pi(s^*,r))$,
where $s^*=(s^*_1(r),s^*_2(r))$ denotes the optimal position thresholds given rank
levels $r=(r_1,r_2)$ and
the number of candidates $n$, obtained by optimizing the
$P_{\rm DLP}^{(n)}(\pi(s,r))$ with respect to $s=(s_1,s_2)$ (see Section~\ref{s:num}).

\begin{figure}[htb]
\begin{center}
\includegraphics[width=8cm]{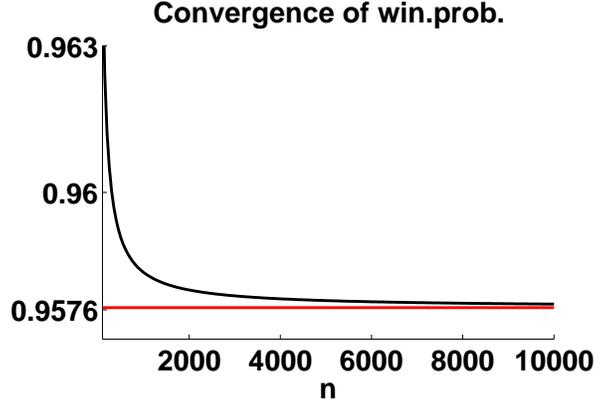}
\caption{\textit{Optimal finite and asymptotic winning probabilities for
double-level policies with $b=5, r_1=2, r_2=6$.}}
\label{f:convb10r12r26}
\end{center}
\end{figure}
\section{Numerical Results}\label{s:num}
We can find numerically the optimal single-level policy
for a given number of candidates $n$, and a given worst allowable rank $b$,
in a two-step approach as:
\[
\max_{r=1,\ldots,b}\; \max_{s=r,\ldots,n-1} P_{\rm SLP}(\pi(s,r)).
\]
Thus, in the first step, we fix also a rank level $r$ (between $1$ and $b$).
The function $\{r,\ldots,n-1\}\to P_{\rm SLP}(\pi(\cdot,r))$ is unimodal
concave (this follows after a marginal analysis), and thus
we can solve numerically for the optimal position threshold $s^*=s^*(r)$,
and the associated winning probability $P_{\rm SLP}(\pi(s^*,r))$.
The second step is simply a complete enumeration
to determine
\[
\max \{ P_{\rm SLP}(\pi(s^*,r)): r=1,\ldots b\}.
\]
However, it can be shown that the function
$\{1,\ldots,b\}\to P_{\rm SLP}(\pi(s^*(\cdot),\cdot))$
is unimodal, which yields a shortcut in the second step.
To check our numerical results, we have constructed an alternative
method to find the optimal position threshold $s^*(r)$, given $n,b,r$,
namely by dynamic programming (see the Appendix for the details).
\par\bigskip\noindent
Similarly, in the case of double-level policies, we have constructed a two-step approach, where the first step finds
the optimal position thresholds $s_1^*=s_1^*(r_1,r_2)$ and $s_2^*=s_2^*(r_1,r_2)$ for any given pair of rank levels
$(r_1,r_2)$, and its associated winning probability $P_{\rm DLP}(\pi(s^*,r))$ (vector notation for $s$ and $r$). Then a
straightforward search procedure determines
\[
\max_{r_1=1,\ldots,b-1}\; \max_{r_2=r_1+1,\ldots,b} P_{\rm DLP}(\pi(s^*,r)).
\]
Finally, as mentioned in the introductory section,
dynamic programming can be applied easily to
obtain the optimal (multi-level) policy \cite{franksamuels,guseinzade}.
\par\bigskip\noindent
Table \ref{t:ren} gives  the relative errors of the winning probabilities
of the optimal single and double-level policies for $n=100, 250$, and $n=1000$,
and for $b=5,10,\ldots,25$, relatively to the corresponding optimal multi-level policies.
The double-level policy gives extremely small errors for larger $b$,
up to very large population sizes $n$. Also we notice that
the errors (for a given $b$) increase slightly as $n$ increases.
\begin{table}[htb]
\begin{center}
\caption{\textit{Relative errors (\%) of the optimal single- and double-level policies.}}
\label{t:ren}
\medskip
\begin{tabular}{l | r r r | r r r}
& \multicolumn{3}{c}{single-level} & \multicolumn{3}{c}{double-level}\\
& $n=100$ & $n=250$ & $n=1000$ & $n=100$ & $n=250$ & $n=1000$\\
\hline
$b=5$ & 10.630 & 10.854 & 10.965 & 3.286  & 3.331 & 3.354\\
$b=10$ & 5.262 & 5.674 & 5.876 & 1.702 & 1.841 & 1.911 \\
$b=15$ & 2.095 & 2.467 & 2.658 & 0.568 & 0.686 & 0.746\\
$b=20$ & 0.739  &0.996  & 1.131 & 0.155 & 0.221 & 0.258\\
$b=25$ & 0.239 & 0.381 & 0.464 & 0.036 & 0.066 & 0.084
\end{tabular}
\end{center}
\end{table}
\noindent
Finally, we have computed the optimal asymptotic winning probabilities
of the level policies:
\[
\max_{r=1,\ldots,b}\; P_{\rm SLP}^{(\infty)}(\pi(\alpha^*,r)),
\]
where $\alpha^*=\alpha^*(r)$ is the associated proportional rank level given $r$, obtained by the procedure elaborated
in Section \ref{s:slp}. Similarly, for the double-level policies
\[
\max_{r_1=1,\ldots,b-1}\; \max_{r_2=r_1+1,\ldots,b} P_{\rm DLP}^{(\infty)}(\pi(\alpha^*,r))
\]
yields the optimal asymptotic winning probabilities. Table \ref{t:asres} summarizes our computations for a range of
$b$-values. Also we included the asymptotic results of the  optimal (full) policy, given in Frank and Samuels
\cite{franksamuels} ($t_1=\lim_{n\to\infty} s_1^*/n$ for the optimal position threshold). Again we see how accurate the
approximations of the level policies are. Notice that the $\alpha^*$ thresholds are not
monotone in $b$, this is due to the discrete character of the
levels $r_1$ and $r_2$.

\begin{table}[htb]
\begin{center}
\caption{\textit{Asymptotics of the optimal multi-level, single- and double-level policies.}}
\label{t:asres}
\medskip
\begin{tabular}{r | c c | r c c| r r c c c c}
$b$ & $t_1$ & $P(\pi)$
& $r$ & $\alpha$ & $P_{\rm SLP}(\pi)$
& $r_1$ & $r_2$ & $\alpha_1$ & $\alpha_2$ & $P_{\rm DLP}(\pi)$\\
\hline
 5 & 0.3255 & 0.860347 &  3 & 0.5046 & 0.765697 & 1 &  4 & 0.2996 & 0.6559 & 0.831420\\
10 & 0.3129 & 0.976530 &  4 & 0.4692 & 0.918487 & 2 &  6 & 0.3630 & 0.6446 & 0.957643\\
15 & 0.3068 & 0.995902 &  6 & 0.5152 & 0.968786 & 3 &  9 & 0.3960 & 0.6822 & 0.988265\\
20 & 0.3031 & 0.999271 &  7 & 0.4990 & 0.987504 & 4 & 12 & 0.4164 & 0.7051 & 0.996561\\
25 & 0.3006 & 0.999869 &  9 & 0.5270 & 0.994938 & 5 & 14 & 0.4304 & 0.6965 & 0.998961
\end{tabular}
\end{center}
\end{table}
\section{Conclusion}\label{s:con}

For the considered generalized secretary problem of selecting one of the $ b $ best out of a
group of $ n $ we have obtained closed expressions for the probability of success for all
possible single- and double-level policies. For any given finite values of $ n $ and $ b $ these
expressions can be used to obtain the optimal single-level policy
respectively optimal double-level policy in a straightforward manner. Moreover, asymptotically
for $ n \rightarrow \infty $ we have also obtained closed expressions for the winning probability
for relevant families of single-level and
double-level policies. Optimizing this expression for the family of single-level policies an asymptotic optimal rank
level $ r $ and corresponding optimal position threshold fraction $ \alpha^* $ and asymptotic winning probability are
easily obtained. Similarly we have done such asymptotic analysis and optimization for the relevant family of
double-level policies. Both for the single-level and double-level policies we confirmed numerically for $ b =5 $ that
the winning probabilities for optimal finite and double level policies for finite values of $ n $ converge if $ n $
increases to the (respectively single-level and double-level) optimal asymptotic winning probabilities.
\par\bigskip\noindent
Finally, we computed for varying $ b $ and $ n $ the optimal single-level and double-level policies and corresponding
winning probabilities and compared the results to the overall optimal policy which is determined by $ b $ position
thresholds. We found that the single-level policies and especially the double-level policies perform nearly as well as
the overall optimal policy. In particular for a generalized secretary problem with a larger value of $ b $ applying the
optimal single-level or double-level policy could be considered, because implementation of the overall optimal policy
using $ b $ different thresholds is unattractive compared to using only one or two thresholds for implementing the
policy. Besides for large $ b $ the gain in performance of the overall optimal policy over the optimal double-level
policy is very small.

%
%%%%%%%%%%%%%%%%%%%            references        %%%%%%%%%%%%%%%%%%%%%%%%
\large
\renewcommand{\baselinestretch}{1.0}

\large
\renewcommand{\baselinestretch}{1.3}
\normalsize

\section*{Appendix: Dynamic Programming}
The dynamic programming method might be applied to find numerically
the optimal single-, double- and multiple-level (`full') policies.
Here, we summarize the algorithm for the single-level policy; it
is straightforward how to generalize the algorithm to the double-level,
and the multiple-level cases.
\par\bigskip\noindent
Define the single-level policy with threshold $s$ and level $r$,
denoted $\pi(s,r)$,  by its actions
\begin{align*}
i\leq s: & \; a_i(k) = 1 \; \mbox{for all $k$}\\
s+1\leq i\leq n-1: & \; a_i(k)
= \begin{cases}
0, & \; \text{if}\; k\leq r\\
1, & \; \text{if}\; k>r;
\end{cases}\\
i = n: & \; a_n(k)=0 \; \mbox{for all $k$},
\end{align*}
where 1 means to continue, and 0 means to stop and select this candidate.
We restrict to $1\leq r\leq b$.
Denote by $P_{\rm SLP}(\pi(s,r))$ the probability of winning when $\pi(s,r)$ is applied.
Given level $r$ we determine the optimal threshold $s^*(r)$,
defined by
\[ s^*(r) = \arg\max_s\; P_{\rm SLP}(\pi(s,r)).\]
We use dynamic programming to find it. Define for $i=1,\ldots,n-1$ the value $f_i(1)$ to be the maximal probability of
winning when $Y_i\leq r$ is observed, and $f_i(2)$ to be the maximal probability of winning when $Y_i> r$.
The optimality equations are:
\begin{align*}
f_{n-1}(1) &= \max\, \left\{ \underbrace{\Prob(X_{n-1}\leq b\,|\,Y_{n-1}\leq r)}_{a=0},\,
\underbrace{\Prob(Y_n\leq b)=\frac{b}{n}}_{a=1}\right\}\\
f_{n-1}(2) &= \frac{b}{n};
\end{align*}
for $i=n-2,n-3,\ldots,r$
\begin{align*}
f_{i}(1) &= \max\, \{ \underbrace{\Prob(X_i\leq b\,|\,Y_i\leq r)}_{a=0},\,
\underbrace{f_i(2)}_{a=1}\}\\
f_{i}(2) &= \frac{r}{i+1}f_{i+1}(1) + \frac{i+1-r}{i+1}f_{i+1}(2),
\end{align*}
and for $i=r-1,r-2,\ldots,1$ (since then surely $Y_i\leq r$ and $Y_{i+1}\leq r$):
\begin{align*}
f_{i}(1) &= \max\, \{ \underbrace{\Prob(X_{i}\leq b)}_{a=0},\,
\underbrace{f_{i+1}(1)}_{a=1}\}\\
f_{i}(2) &= \mbox{not defined}.
\end{align*}
One can show that the result of this DP recursion is indeed a SLP by setting $s^*(r) = \max \{i: a^*_i=1\}$. Moreover,
note that probabilities $ \Prob(X_i\leq b\,|\,Y_i\leq r) $ occuring in the optimality equations can easily be obtained,
for example by applying Lemma \ref{l:hypprob} and Bayes' rule.
\end{document}